\documentclass{amsart}

\usepackage{setspace}
\usepackage{bbm}
\usepackage{graphicx}
\usepackage{mystyle}
\usepackage{proof}
\usepackage{fitch}
\usepackage{amsbsy,eucal}
\usepackage{amsmath}
\usepackage{amsfonts}
\usepackage{latexsym}
\usepackage[all]{xy}
\newlength{\mathfrwidth}
\newsavebox{\mathfrbox}
\newenvironment{mathframe}
    {\begin{lrbox}{\mathfrbox}\begin{minipage}{\mathfrwidth}\begin{center}}
    {\end{center}\end{minipage}\end{lrbox}\noindent\fbox{\usebox{\mathfrbox}}}
    
\makeindex

\newcommand{\cal}{\mathcal}
\newcommand{\floor}[1]{\lfloor#1\rfloor}

\theoremstyle{definition}

\numberwithin{equation}{section}

\begin{document}

\title{Majority Digraphs}


\author{Tri Lai}
\address{Institute for Mathematics and its Applications, University of Minnesota, Minneapolis MN 55455 USA.}
\curraddr{}
\email{tmlai@ima.umn.edu}
\thanks{}

\author{J\"{o}rg Endrullis}
\address{Vrije Universiteit Amsterdam,
1081 HV Amsterdam, The Netherlands,
 and
Department of Mathematics, Indiana University, Bloomington IN 47405 USA}
\curraddr{}
\email{j.endrullis@vu.nl}
\thanks{}

\author{Lawrence S.~Moss}
\address{Department of Mathematics, Indiana University, Bloomington IN 47405 USA}
\curraddr{}
\email{lsm@cs.indiana.edu}
\thanks{The first author was partially supported  by the Institute for Mathematics and its Applications with funds provided by the National Science Foundation (\#DMS-0931945). The third author was partially supported by a grant from the Simons Foundation (\#245591).}


\date{}

\dedicatory{}

\commby{}

\begin{abstract}
A \emph{majority digraph}
 is a finite simple digraph  $G=(V,\to)$
such that there exist finite sets $A_v$ for the  vertices $v\in V$
with the following property:
  $u\to v$ if and only if   ``more than half of the $A_u$ are $A_v$''.
That is, $u\to v$ if and only if  $ |A_u \cap A_v | > \frac{1}{2} \cdot |A_u|$.
We characterize the majority digraphs
 as the digraphs with the property that every directed cycle has a reversal.
 If we change $\frac{1}{2}$ to any real number $\alpha\in (0,1)$, we obtain the same class of digraphs.
  We apply
the characterization result   to obtain a result on the logic of assertions ``most $X$ are $Y$'' and the
standard connectives of propositional logic.
\end{abstract}

\maketitle

\section{Introduction}
This paper poses a problem in combinatorics
coming from logic.   For finite sets $X$ and $Y$, we say that \emph{most $X$ are $Y$} if
$|X \cap Y | > \frac{1}{2} | X|$.  If  \emph{most $X$ are $Y$}, then it need not be the case
that  \emph{most $Y$ are $X$}, but it would follow (trivially) that  \emph{most $X$ are $X$}.
If  \emph{most $X$ are $Y$} and  \emph{most $Y$ are $Z$}, then it need not be the case
that  \emph{most $X$ are $Z$}.
People with a background in logic
would ask questions  about sound  inferences involving \emph{most}:  are there any interesting
sound inferences at all?
Is there
a characterization of the collection of all sound inferences?
What is the complexity of that collection?
  We shall formulate
 the inference question precisely  and answer it in Section~\ref{section-logic}
near the
end of this paper.
The solution hinges on a result in elementary combinatorics, and this
result is the main mathematical contribution of this paper.

If $V$ is any finite set, and $A_v$ is a finite set for $v\in V$, then we obtain a digraph $G=(V,\to)$
in a natural way:  $u\to v$ iff \emph{most $A_u$ are $A_v$}.
We are only interested in digraphs without self-loops, so when we write $u\to v$ in this paper,
we tacitly assume that $u$ and $v$ are different.
A
\emph{majority digraph} is a finite  digraph
isomorphic to some digraph of this form.
 The characterization of sound inferences involving \emph{most}
  boils down to the characterization of majority digraphs.
We next state our  main result.

A \emph{two-way edge} in a digraph is just an edge $u\to v$ in the digraph such that also $v\to u$.
A \emph{one-way edge}  is an edge $u\to v$ such that  $v \not\to u$.
If $G$ is a  majority digraph
via the sets $A_v$,
 and if there is a one-way edge from $u$ to $v$, then $|A_v | > |A_u|$.
Thus $G$ cannot have  \emph{one-way cycles}: there are no
paths \begin{equation}
v_1 \to v_2 \to \cdots \to v_n = v_1
\label{vs}
\end{equation}
 such that
for $1\leq i < n$,
$v_{i+1}\not\to v_{i}$.
(There may be cycles with two-way edges.)
This point was noticed by
Chloe Urbanski~\cite{Urbanski}, and she conjectured that
the absence of one-way cycles characterizes majority digraphs.
This turns out to be true, and it is our main result.

More generally, for any $\alpha\in (0,1)$, we say that $G=(V,\to)$ is a \emph{proportionality $\alpha$-digraph}
if
there exist finite sets $A_v$ for $v\in V$
with the following property:
$$u\to v \quadiff  |A_u \cap A_v | >\alpha\cdot |A_u|.
$$
So a majority digraph is a  proportionality $\frac{1}{2}$-digraph.
Our second main result is that the characterization  result for
majority digraphs holds as well for
 proportionality $\alpha$-digraphs, for all $\alpha\in (0,1)$.

\subsection{Contents}
Section~\ref{section-one-way} has a  very general (and very easy)
representation result on digraphs with the property that
every directed cycle  has a \emph{reversal}.
 That is, for every path as in (\ref{vs}) there is some $1 \leq i < n$
such that $v_{i+1} \to v_i$.
(This is the same as having no  one-way cycles.)
In Section~\ref{section-rational}, we show that this condition characterizes majority digraphs;
indeed, it characterizes $\alpha$-proportionality digraphs for all rational $\alpha$.
Then in Section~\ref{section-real-alpha} we obtain the result for all real $\alpha\in (0,1)$.
The work on
rational $p/q$ in  Section~\ref{section-rational} is not merely a special case of the later results on
 real $\alpha$ in Section~\ref{section-real-alpha}.
The point is that to carry out our construction for irrational $\alpha$ necessitates using much larger
sets than the construction
 when $\alpha$ is a rational number.   Put differently, the result in Theorem~\ref{theorem-real} is a generalization
of the result in  Theorem~\ref{theorem-rational}, but the construction in Theorem~\ref{theorem-rational}
gives better bounds for the digraphs it constructs.

We conclude the paper in Section~\ref{section-logic}
by returning to the matter in logic with which we began.
Section~\ref{section-logic} may be read after Section~\ref{section-rational}.

\subsection{Preliminary}
\label{section-one-way}

For a fixed number $n$, an \emph{appropriate pair} is a pair $(S, T)$ such that
\begin{enumerate}
\item $S$ is a set of unordered pairs $\set{i,j}$ from the set of numbers $\set{1,\ldots, n}$.
\item $T$ is a set of ordered pairs $(i, j)$  from $\set{1,\ldots, n}$.
\item If $(i,j)\in T$, then $i < j$.
\item If $(i, j) \in T$, then $\set{i,j}\notin S$.
\end{enumerate}
Further,  every appropriate pair determines a digraph $G_{S,T}$.
The vertices of $G_{S,T}$ are the points of $\set{1,\ldots, n}$, and we put $i \to j$
iff either $\set{i,j}\in S$ or $(i, j) \in T$.

\begin{proposition}
Let $G$ be a digraph on $n$ vertices  with no one-way cycle.
Then there is an appropriate pair $(S,T)$ such that $G$ is isomorphic to $G_{S,T}$.
\label{proposition-appropriate}
\end{proposition}

\begin{proof}
First, we may assume that the vertices of $G$ are $\set{1, \ldots, n}$.
We may list these in topological order.
So we have a sequence  $1, \ldots, n$, with the property  that if
$i\to j$ but $j\not\to i$, then $i < j$.   This is due to the assumption that there be no one-way cycle.
We can take $S$ to be the set of pairs corresponding to the two-way edges,
and $T$ the one-way edges.
\end{proof}

\section{The case when $\alpha$ is a rational number $p/q$}
\label{section-rational}

In this section, we represent digraphs with no one-way cycle as proportionality $p/q$-digraphs
for all natural numbers $0 < p < q$.   Taking $p/q = 1/2$, we see that digraphs
with no one-way cycles are majority digraphs.

The reader may wish to consult a worked example which we present in Section~\ref{section-example} below.

\rem{The more general case of an arbitrary real number $\alpha \in (0,1)$  is more difficult, and we
turn to it in Section~\ref{section-real-alpha} below.
}

For a sequence of sets $A_1, \ldots, A_n$
and for $1 \leq i < j \leq n$,  we write $A_i \sqcap A_j$ for $(A_i \cap A_j)\setminus \bigcup_{k\neq i, j} A_k $.
We call this the \emph{private intersection} of $A_i$ and $A_j$.
This is an example of what is sometimes called a \emph{zone} in a Venn diagram.

\begin{lemma}   Let
$p\leq q$ be natural numbers.
For all $n$, there are sets $B_1, \ldots, B_n$ such that
\begin{enumerate}
\item $|B_i |= p q^{n-1}$.
\item For $i \neq j$, $|B_i \cap B_j | =  p^2 q^{n-2}= \frac{p}{q}|B_i|$.
\item For $i\neq j$, $| B_i \sqcap B_j| = p^2 (q-p)^{n-2}$.
\end{enumerate}
\label{lemma-binary-digits}
\end{lemma}

\begin{proof}
Consider $S = \set{1, \ldots, q}$.  Let $$B_i
\quadeq
\set{(s_1 s_2 \cdots s_i \cdots s_n) \in  \set{1, \ldots, q}^n : s_i \in \set{1, \ldots, p}
}.$$
The first two parts are obvious.
$B_i \sqcap B_j$ is the set of sequences $(s_1 s_2 \cdots s_i \cdots s_n)$
so that $s_i$ and $s_j$ belong to $\set{1, \ldots, p}$,
and the other entries do not belong to this set.
\end{proof}

We shall use the following elementary observation:
\begin{equation}
\displaystyle{\frac{p + ar}{q + ar + s}} > \displaystyle{\frac{p}{q}}
\quadiff a >
\displaystyle{\frac{ps}{r(q-p)}}
\ .
\label{elementary}
\end{equation}

\begin{theorem}  Let $G$ be a digraph on $n$ vertices  with no one-way cycle.
Let $p < q$ be positive natural numbers.
Then $G$ is a proportionality $p/q$-digraph.
\label{theorem-rational}
\end{theorem}

\begin{proof}
By Proposition~\ref{proposition-appropriate}, we find
an appropriate pair $(S,T)$
such that $G$ is isomorphic to $G_{S,T}$.

For our $p$ and $q$,
let $B_1, \ldots, B_n$ be as in Lemma~\ref{lemma-binary-digits}.
We shall modify these sets in several steps to obtain sets
 $A_1, \ldots, A_n$
such that
the following hold:
\begin{itemize}
\item[(a)]  If $\set{i,j}\in S$, then  $|A_i\cap A_j| > \frac{p}{q} |A_i|$ and $|A_i\cap A_j| > \frac{p}{q} |A_j|$.
\item[(b)] If $i < j$ and $(i,j)\in T$, then  $|A_i\cap A_j| > \frac{p}{q} |A_i|$
but $|A_i\cap A_j| \leq \frac{p}{q} |A_j|$.
\item[(c)] If $i < j$ and $(i,j)\notin T$, then
$|A_i\cap A_j| \leq \frac{p}{q} |A_i|$ and
$|A_i\cap A_j| \leq \frac{p}{q} |A_j|$.
\end{itemize}

Let $a$  and $m$ be natural numbers which are large enough so that
the following hold:
\begin{eqnarray}
(q-p )a & > & q\:;
\label{assumption-a}  \\
mp^2 (q -p)^{n-2} &  >  &  ap n.   \label{assumption-m1}
\end{eqnarray}
To begin,  take $m$ copies of all points in all sets $B_i$.
(That is, let $A_i = B_i\times \set{1,\ldots, m}$.)
This arranges that $|A_i| = m p q^{n-1}$ for all $i$,
and for $i \neq j$, $|A_i\cap A_j| = m p^2 q^{n-2}$,
and
 $|A_i\sqcap A_j| > ap n$.
We have used (\ref{assumption-m1}) here.

Add  a  single set $C$ of $a pn $
 points simultaneously to all $A_i$.
 That is, we have $C \subseteq \bigcap_i A_i$.
 (We are going to continue to call the sets $A_i$ rather than change the notation.)
This adds $  a pn$   points to all intersections $A_i \cap A_j$, so now these sets have size
$mp^2 q^{n-2 }+ a pn$.
But this addition leaves all private intersections $A_i\sqcap A_j$
unchanged.

Next, for each $i$, add  $qi$  fresh points to $A_i$.
In this step, we add different points
to the different $A_i$.  This step does not change (private) intersections,
it only increases the sizes of the sets.

When $\set{i,j}\in S$, the rest of our construction will not
alter the intersection $A_i\cap A_j$ or the sizes of $A_i$ and $A_j$.
 So in this case, we shall have at the end that
$$
\begin{array}{lclclcl}
\displaystyle{\frac{|A_i \cap A_j |}{|A_i | } } & = &
\displaystyle{ \frac{mp^2 q^{n-2 }+ a pn }{m p q^{n-1}  +apn + qi} }& \geq &
\displaystyle{ \frac{mp^2 q^{n-2 }+ a pn }{m p q^{n-1} + apn + q n} }& > &
\displaystyle{ \frac{p}{q}} \  .
\end{array}
$$
   We have used (\ref{elementary}) with $r = pn$ and $s = qn$,
   and also the assumption (\ref{assumption-a}) on $a$.
   Similarly, $|A_i\cap A_j | > \frac{p}{q} |A_j|$.

   We are left with two cases:  (a)  $i < j$  and $(i,j)\in T$ (and thus $\set{i,j}\notin S$)
and (b)  $i < j$  and $(i,j)\notin T$ and $\set{i,j}\notin S$.
For the pairs of the first type, we make a certain adjustment to the sets we have,
removing points from $A_i \sqcap A_j$
and returning them as separate copies in the two sets.
(So this type of  adjustment does not change the size of any $A_i$, but it decreases
the sizes of the intersections $A_i \cap A_j$.)
   It will turn out that the number of points
which we remove in this case depends on $i$.    And for the second type,
we remove \emph{all} the points in $A_i \sqcap A_j$ and return them as
separate copies in the two sets.
All of these adjustments  of either type
 may be carried out at the same time, and there is no need
 to order them.

The   case (b)
of $i < j$  and  also $(i,j)\notin T$ and $\set{i,j}\notin S$
is easier to handle, so let us look at this first.
The  private intersection $A_i \sqcap A_j$ has size $mp^2 (q-p)^{n-2}$.  Let us call this number $z$.
Take the entire private intersection
and remove it, returning separate copies of the same  size  $z$ to  $A_i$ and to $A_j$.
The removal decreased the size of the intersection $A_i\cap A_j$ by $z$.
By  (\ref{assumption-m1}), $apn - z  < 0$.   We calculate:
$$
\begin{array}{lclclcl}
\displaystyle{\frac{|A_i \cap A_j |}{|A_i | } } & = &
\displaystyle{ \frac{mp^2 q^{n-2 }+ a pn - z }{m p q^{n-1}  +apn + qi} }&<  &
\displaystyle{ \frac{mp^2 q^{n-2 } }{m p q^{n-1} } }&= &
\displaystyle{ \frac{p}{q}} \ .
\end{array}
$$
Similarly, $|A_i \cap A_j |/{|A_j | }  < p/q$.

Finally,  let   $i < j\leq n$  and $(i,j)\in T$.
The idea is to do something similar to
what we did in the last paragraph:
  remove a certain number of points from the private intersection $A_i\sqcap A_j$
  and then return the same number of points in separate copies to $A_i$ and $A_j$.
  But we want to remove a \emph{proper} subset of points,
  so that
more than $p/q$ of the $A_i$ are $A_j$, but  at most $p/q$ of the $A_j$ are $A_i$.
By (\ref{assumption-a}),
  $$
  \begin{array}{lclclcl}
  \displaystyle{\frac{p}{q}} apn + pi & <  & \displaystyle{\frac{p}{q}} apn + pn & < & apn\ .
   \end{array}
   $$
Let
$$\begin{array}{lcl}
c & = &  \displaystyle{\left\lceil\left(\frac{q - p}{q}\right)apn \right\rceil} - pi  - 1 \ .
\end{array}
$$
This has the property that
$$\begin{array}{lclcl}
  \displaystyle{\frac{p}{q}} apn + p i
  & <  & apn  - c  & \leq  &   \displaystyle{\frac{p}{q}} apn + p i  + 1
  \ .
\end{array}
$$
Note that $c < apn$, and as we have seen, $apn < |A_i \sqcap A_j|$.
We remove $c$ points from $A_i\sqcap A_j$ and return them  separately to $A_i$ and $A_j$.
So the intersection $A_i\cap A_j$ has size $ m p^2 q^{n-2}   + apn  - c  $.
To check that this works, we calculate:
  $$\setstretch{2}
  \begin{array}{lclcl}
 \displaystyle{\frac{p}{q}}( m p q^{n-1} + apn + qi) & = & m p^2 q^{n-2} +  \displaystyle{\frac{p}{q}} apn + p i
\\
 & <  &  m p^2 q^{n-2}   + apn  - c  \\
& \leq &   m p^2 q^{n-2} +  \displaystyle{\frac{p}{q}} apn + p i  + 1    \\
& \leq &  m p^2 q^{n-2} +  \displaystyle{\frac{p}{q}} apn + p j  &\mbox{(since $i < j$ and $1\leq p$)} \\
 & = &  \displaystyle{\frac{p}{q}}( m p q^{n-1} + apn + qj).
 \end{array}
 $$
  That is,
  $$\begin{array}{lclclcl}
\displaystyle{\frac{p}{q}} |A_i| & <    & | A_i \cap A_j | & \leq &  \displaystyle{\frac{p}{q}} |A_j|.
  \end{array}
  $$
We have achieved our goals (a), (b), and (c).  This completes the proof.
\end{proof}

\begin{remark}
Let us see how many points are needed to exhibit
a digraph without one-way cycles as  a $\frac{1}{2}$-digraph.
Suppose that $G$  has
 $n$ vertices and $e$ edges.
 We would like to know the size of $\bigcup_i A_i$
 using
the method of this section.
We have $p = 1$ and $q=2$,
and in Lemma~\ref{lemma-binary-digits}, $\bigcup_i A_i$ has size $2^n$.
Further, we may take  $a = 3$ and $m  = 3n$.
Following the proof, we get a universe of at most $3n (1 + 2^n) + n(n+1) + 3n e$ points.
So we get  $|\bigcup_i A_i| = O(n 2^n)$.
\end{remark}

\subsection{Example}
\label{section-example}
We illustrate all of the ideas in the proof of Theorem~\ref{theorem-rational}
  with an example.
  Consider the digraph $G$ shown in Figure~\ref{figure-graph-example} below.
We thus begin with $n = 4$, $p=1$, and $q = 2$.
The usual order $1 < 2 < 3 < 4$ has the property that if $i \to j$ but $j \not\to i$, then $i < j$.
From the graph, we have
$$\begin{array}{lcl}
S  &  = & \set{\set{1,2}, \set{2,3} }\\
T &  = & \set{ (1,3),  (3,4), (2,4)}  \\
\end{array}
$$
Lemma~\ref{lemma-binary-digits}
gives sets $B_1$, $\ldots$, $B_4$ with the property that all zones in their Venn diagram
have size $1$.    This is the first diagram in Figure~\ref{figure-Venn}.
Continuing,  we take $a$ to be $3$; this is the minimum number so that
(\ref{assumption-a}) holds.
Thus $a p n = 12$, and $m = 13$ is the smallest so that
(\ref{assumption-m1}) holds.    We continue by taking $13$ copies of all points,
and we rename the sets $A_1$, $\ldots$, $A_4$.

Next, we add $apn = 12$ points to $\bigcap  A_i$.
This is shown in the left-hand Venn diagram on the second row.
Continuing, we add $2$ points to $A_1$, $4$ to $A_2$, $6$ to $A_3$, and $8$ to $A_4$.
When we add to $A_i$ in this step, we are adding to $A_i \setminus \bigcup_{j\neq i} A_j$.
This is what we mean by a ``private'' addition.
This is shown in the right-hand Venn diagram on the second row.

We can check at this point that for $\set{i,j}\in S$,
$|A_i \cap A_j| > \frac{1}{2} |A_j|$.
For example, when $j = 3$, and $i = 2$,
we have $|A_3 \cap A_2| = 64$, and $A_2  = 120$.

Next, $(1,4)\notin T$ and $\set{1,4}\notin S$.
So we take $A_1 \sqcap A_4$,
remove all $13$ of its points, and then add $13$ points privately to $A_1$,
and finally $13$ other points privately to $A_4$.
The picture is the left diagram on the bottom row.
Then $|A_1 \cap A_4| = 51$.   This is smaller than $\frac{1}{2} |A_1| = 59 $
and also smaller than $\frac{1}{2} |A_4| =  62$.

It remains to take care of the pairs in $T$:  $(1,3)$,  $(3,4)$, and  $(2,4)$.
For $(1,3)$,
the value of $c$ is $\lceil\frac{1}{2}(12)\rceil - 1(1) - 1 = 4$.
 We remove $4$ points from $A_1 \sqcap A_3$ and
return them privately to $A_1$ and $A_3$.    This is how we get $|A_1 \sqcap A_3| = 13-4 = 9$
at the end.
For $(3,4)$, $c = \lceil\frac{1}{2}(12)\rceil - 1(3) - 1 = 2$.  We remove $2$ points from $A_3 \sqcap A_4$ and return them
privately to $A_3$ and $A_4$.   And at the end,  $|A_3\sqcap A_4| = 13-2 = 11$.
For $(2,4)$, $c = 3$.   We remove $3$ points from $A_2 \sqcap A_4$ and return them
privately to $A_2$ and $A_4$.   Thus  $|A_2\sqcap A_4| = 13-3 = 10$.
Then
totaling up the three private additions mentioned in this paragraph gives the sizes of the
sets $A_i \setminus \bigcup_{j\neq i} A_j$ for $i = 1, \ldots 4$.
The final Venn diagram is shown at the bottom right  of Figure~\ref{figure-Venn}.
It exhibits the digraph in Figure~\ref{figure-graph-example} as a majority digraph.

\begin{figure*}[tb]
\begin{mathframe}
\centering
 \makebox[\textwidth]{\includegraphics[width=2cm]{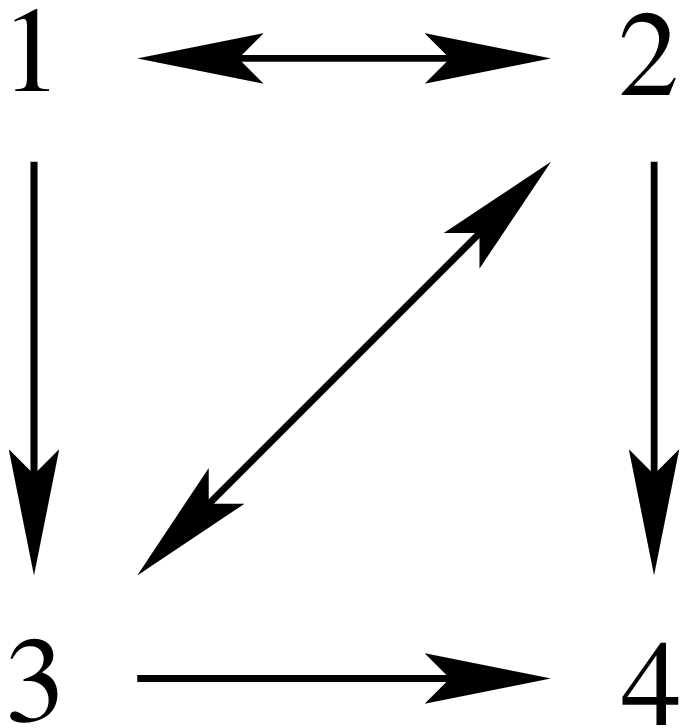}}
\caption{A digraph without one-way cycles used to illustrate the steps in Theorem~\ref{theorem-rational}.}
\label{figure-graph-example}
\end{mathframe}

\bigskip

\begin{mathframe}
\centering
\smallskip
 \makebox[\textwidth]{\includegraphics[width=.55\paperwidth]{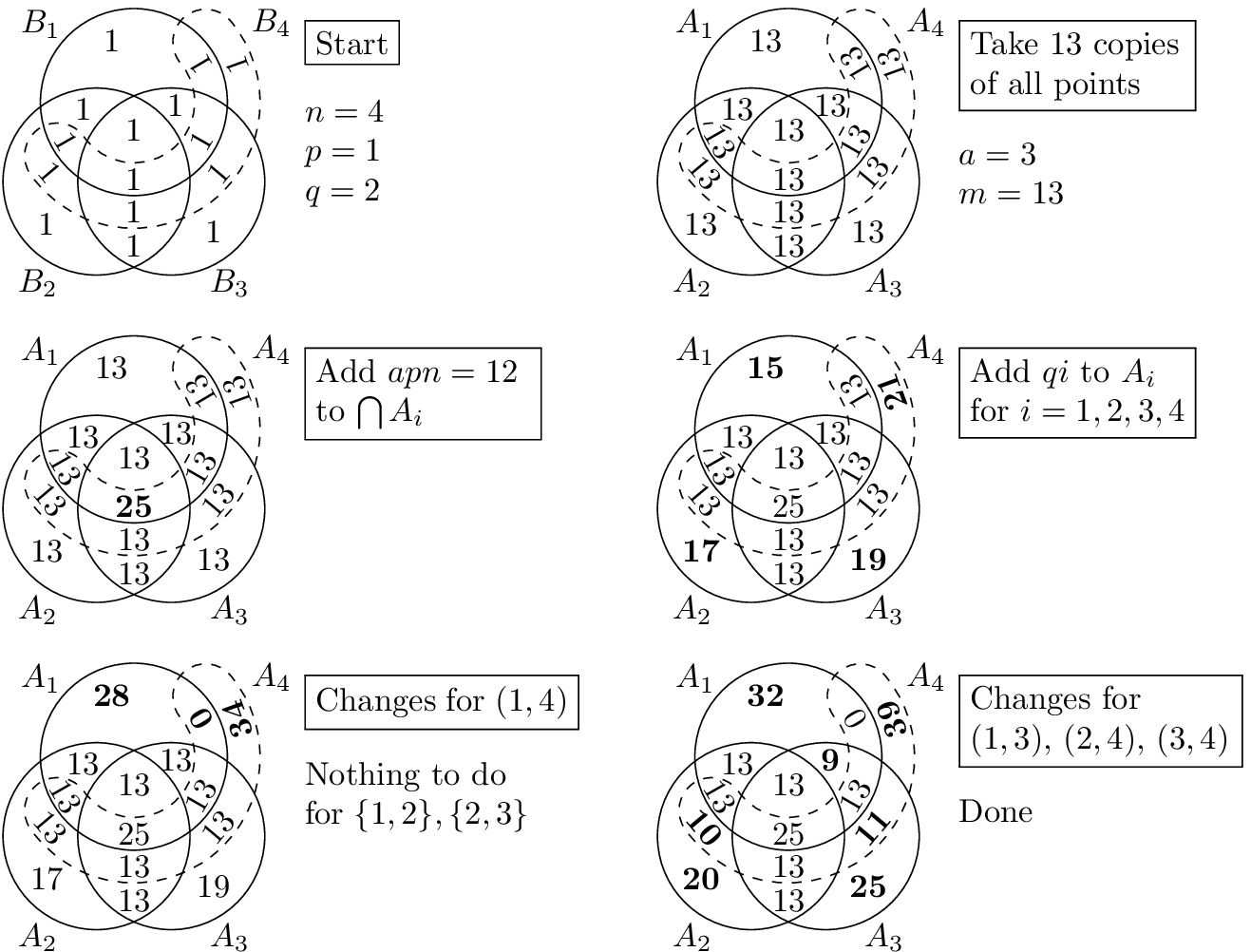}}
\caption{Six Venn diagrams illustrating the construction.  These are explained in detail in Section~\ref{section-example}.}
\label{figure-Venn}
\end{mathframe}
\end{figure*}

\section{Proportionality $\alpha$-digraphs}
\label{section-real-alpha}

\newcommand{\tups}{\Delta}
\newcommand{\powern}{P}

This section proves
Theorem~\ref{theorem-real}, a  generalization of  Theorem~\ref{theorem-rational}.
We begin with a few auxiliary definitions.
Fix a natural number $n$, and
let $\powern$  be the set of all subsets of  $\{1,\ldots,n\}$.
We define
\begin{align*}
  \tups(i,\ldots,i_\ell)  =   \tups(\{i,\ldots,i_\ell\}) = \{\; I \in \powern \mid \; \{i,\ldots,i_\ell\} \subseteq I \}.
\end{align*}
As the notation indicates, in this discussion we frequently omit set braces in the arguments of $\Delta$.
An  \emph{$n$-size function f} is a function $f : \powern \to \mathbb{R}_+$,
and we associate to $f$ the function $f^* : \powern \to \mathbb{R}_+$  defined by:
\begin{align}
  f^*(I) &= \sum_{J \in \tups(I)} f(J) \label{f:star}
\end{align}
for all $I \in \powern$.
The intuition is that $f(I)$ indicates the size of the private intersection
\begin{align*}  \bigcap_{i\in I} A_i \setminus \bigcup_{j\notin I} A_j
  \;,
\end{align*}
while $f^*(I)$  indicates the size of the intersection  $ \bigcap_{i\in I} A_i $.

Here are the canonical examples of size functions.

\begin{samepage}
\begin{lemma}\label{lem:eq:share}
  Let $n\in \mathbb{N}$ and $\alpha\in (0,1)$.
  There exists a size function $f : \powern \to \mathbb{R}_+$
  such that
  \begin{enumerate}
    \item $f^*(i) = 1$ for every $1 \le i \le n$.
    \item $f(i,j) = \alpha(1-\alpha)^{n-2}$ for every $1 \le i, j \le n$ with $i \ne j$.
    \item $f^*(i,j) = \alpha$ for every $1 \le i, j \le n$ with $i \ne j$.
  \end{enumerate}
\end{lemma}
\end{samepage}

\begin{proof}
  We define $f : \powern \to \mathbb{R}_+$ by:
  \begin{align*}
    f(I) &= \alpha^{|I|-1} (1-\alpha)^{n-|I|}
  \end{align*}
  for all $I \subseteq \powern$.
  Then for $1 \le i \le n$ we have
  \begin{align*}
    f^*(i)
     = \sum_{J \in \tups(i)} f(J)
    &= \sum_{J \in \tups(i)} \left( \alpha^{|J|-1} (1-\alpha)^{n-|J|} \right) \\
    &= \sum_{\ell = 1}^{n} \left( \binom{n-1}{\ell-1} \alpha^{\ell-1} (1-\alpha)^{n-\ell} \right) \\
    &= \sum_{\ell = 0}^{n-1} \left( \binom{n-1}{\ell} \alpha^{\ell} (1-\alpha)^{(n-1)-\ell} \right) \\
    &= \left(\alpha + (1-\alpha)\right)^{n-1} = 1.
  \end{align*}
  For $1 \le i, j \le n$ with $i \ne j$ we get $f(i,j) = \alpha(1-\alpha)^{n-2}$ by definition, and:
  \begin{align*}
    f^*(i,j)
     = \sum_{J \in \tups(i,j)} f(J)
    &= \sum_{J \in \tups(i,j)} \left( \alpha^{|J|-1} (1-\alpha)^{n-|J|} \right) \\
    &= \sum_{\ell = 2}^{n} \left( \binom{n-2}{\ell-2} \alpha^{\ell-1} (1-\alpha)^{n-\ell} \right) \\
    &= \alpha \cdot \sum_{\ell = 0}^{n-2} \left( \binom{n-2}{\ell} \alpha^{\ell} (1-\alpha)^{(n-2)-\ell} \right) \\
    &= \alpha \cdot \left(\alpha + (1-\alpha)\right)^{n-2} = \alpha.
  \end{align*}
  This concludes the proof.
\end{proof}

\begin{lemma}\label{lem:real}
Let $G$ be a digraph on $n$ vertices.
  Let $f : \powern \to \mathbb{R}_+$ be an  $n$-size function such that
  \begin{align*}
    i\to j &\implies  f^*(i,j) > \alpha \cdot f^*(i) \\
    i\not\to j &\implies  f^*(i,j) < \alpha \cdot f^*(i)
  \end{align*}
  for all vertices $i \ne j$ of $G$.
  Then $G$ is a proportionality $\alpha$-digraph.
\end{lemma}

\begin{proof}
   Let $\epsilon > 0$ be a real number small enough such that for all $1 \le i,j \le n$, $i\ne j$:
   \begin{align}
     \frac{f^*(i,j)}{f^*(i)} < \alpha \;&\implies\;
     \frac{f^*(i,j)}{f^*(i)-\epsilon} < \alpha; \label{epsilon:smaller} \\
     \frac{f^*(i,j)}{f^*(i)} > \alpha \;&\implies\;
     \frac{f^*(i,j)-\epsilon}{f^*(i)} > \alpha. \label{epsilon:larger}
   \end{align}
   We choose $N \in \mathbb{N}$ such that
   \begin{align}
     \frac{2^{n-1}}{N} < \epsilon. \label{choice:N}
   \end{align}
   For every $J \in \powern$, let $A(J)$
   be a set of $\floor{ f(J) \cdot N }$  points, with $A(J) \cap A(J') = \emptyset$ for $J \neq J'$.
   Then
   \begin{align}
     f(J) \cdot N - 1 \quad<\quad |A(J)| \quad\le\quad f(J) \cdot N. \label{Avecj:size}
   \end{align}
   For $1 \le i \le n$ we define the set $A_i$ as follows:
   \begin{align*}
     A_i = \bigcup_{J \in \tups(i)} A(J).
   \end{align*}
   Then by~\eqref{f:star} and \eqref{Avecj:size} it follows for $1\le i\le n$ that:
   \begin{align}
     f^*(i)\cdot N - 2^{n-1}
     \;=\;
     \sum_{J \in \tups(i)} (f(J)\cdot N - 1)
     \quad<\quad
     \underbrace{ \sum_{J \in \tups(i)} |A(J)| }_{= |A_i|}
     \quad\le\quad
     f^*(i)\cdot N.
     \label{A_i:estimate}
   \end{align}
Similarly,  for $1 \le i < j \le n$ we have:
   \begin{align}
     f^*(i,j)\cdot N - 2^{n-2}
     \;=\;
     \sum_{J \in \tups(i,j)} (f(J)\cdot N - 1)
     \quad<\quad
     \underbrace{ \sum_{J \in \tups(i,j)} |A(J)| }_{= |A_i \cap A_j|}
     \quad\le\quad
     f^*(i,j)\cdot N.
     \label{A_iA_j:estimate}
   \end{align}
   Now from \eqref{A_i:estimate} and \eqref{A_iA_j:estimate} we conclude:
   \begin{align*}
     \frac{f^*(i,j)\cdot N - 2^{n-2}}{f^*(i)\cdot N}
     \quad<\quad
     \frac{|A_i \cap A_j|}{|A_i|}
     \quad<\quad
     \frac{f^*(i,j)\cdot N}{f^*(i)\cdot N - 2^{n-1}},
   \end{align*}
   and hence
   \begin{align}
     \frac{f^*(i,j) - \epsilon}{f^*(i)}
     \;\le\;
     \frac{f^*(i,j) - \frac{2^{n-2}}{N}}{f^*(i)}
     \quad<\quad
     \frac{|A_i \cap A_j|}{|A_i|}
     \quad<\quad
     \frac{f^*(i,j)}{f^*(i) - \frac{2^{n-1}}{N}}
     \;\le\; \frac{f^*(i,j)}{f^*(i) - \epsilon}. \label{ratio:estimation}
   \end{align}
   Now for all vertices $i \ne j$ of $G$ we have:
   \begin{align*}
     i\to j &\implies  f^*(i,j) > \alpha \cdot f^*(i)
       \quad\stackrel{\text{by \eqref{epsilon:larger} and \eqref{ratio:estimation}}}{\implies}\quad
       \alpha < \frac{f^*(i,j) - \epsilon}{f^*(i)} < \frac{|A_i \cap A_j|}{|A_i|};
     \\
     i\not\to j &\implies  f^*(i,j) < \alpha \cdot f^*(i)
       \quad\stackrel{\text{by \eqref{epsilon:smaller} and \eqref{ratio:estimation}}}{\implies}\quad
       \alpha > \frac{f^*(i,j)}{f^*(i)- \epsilon} > \frac{|A_i \cap A_j|}{|A_i|}.
   \end{align*}
   Hence $G$ is a proportionality $\alpha$-digraph.
\end{proof}

The remainder of the section is concerned constructs of
the appropriate size function for  a digraph $G$ with no one-way cycle.

\begin{theorem}
  If $G$ has no one-way cycle,
  then $G$ is a proportionality $\alpha$-digraph.
  \label{theorem-real}
\end{theorem}
\begin{proof}
  By Proposition~\ref{proposition-appropriate}
  there is an appropriate pair $(S,T)$
  such that $G$ is isomorphic to $G_{S,T}$.
  Without loss of generality, assume that $G = G_{S,T}$.

  We need the following variant of~\eqref{elementary}:
  \begin{align}
    \frac{\alpha + \epsilon}{1 + \epsilon + \delta} > \alpha
    \quadiff \frac{\epsilon(1-\alpha)}{\alpha} > \delta
    \ .
    \label{elementary:2}
  \end{align}

  Let $f : \powern \to \mathbb{R}_+$ be as in Lemma~\ref{lem:eq:share}, where we take $n$ to be the number of
  vertices in $G$, and $\alpha$ as in our theorem.
  Let $\epsilon$ and $\delta$ be defined as follows:
  \begin{align*}
    \epsilon &= \frac{\alpha(1-\alpha)^{n-1}}{2}
  &  \delta &=  \frac{\epsilon(1-\alpha)}{2\alpha}.
  \end{align*}
  Roughly speaking, we add $\epsilon$ commonly to the intersection of all vertices, and
  $\frac{i\delta}{n}$ privately to  vertex $i$ for all $i$.
  Formally, we define a size function $g : \powern \to \mathbb{R}_+$ by:
  \begin{align*}
    g(1,\ldots,n) &= f(1,\ldots,n) + \epsilon \\
    g(i) &= f(i) + \frac{i\delta}{n} &&\text{for all $1\le i\le n$;} \\
    g(I) &= f(I) &&\text{for all $I \in \powern$ with $1 < |I| < n$.}
  \end{align*}
  Then we have  for all $1\le i,j\le n$ with $i\ne j$:
  \begin{align*}
      g^*(i) &= 1 + \epsilon + \frac{i\delta}{n};\\
    g^*(i,j) &= \alpha + \epsilon.
  \end{align*}
  By~\eqref{elementary:2}, we have  that for all $1\le i, j\le n$ with $i\ne j$
  \begin{align}
    \frac{g^*(i,j)}{g^*(i)} &
    = \frac{\alpha + \epsilon}{1 + \epsilon + \frac{i\delta}{n}}
    \ge \frac{\alpha + \epsilon}{1 + \epsilon +\delta}
    > \alpha.
    \label{g:prop}
  \end{align}
  As a consequence, for each $1 \le i < j \le n$ there exists $\gamma(i,j) \in [0,\epsilon]$ such that
  \begin{align}
    \begin{aligned}
    && \frac{g^*(i,j) - \gamma(i,j)}{g^*(i)} &= \frac{\alpha + \epsilon - \gamma(i,j)}{1 + \epsilon + \frac{i\delta}{n}} > \alpha \\
    \mbox{ and } &&     \frac{g^*(i,j) - \gamma(i,j)}{g^*(j)} &= \frac{\alpha + \epsilon - \gamma(i,j)}{1 + \epsilon + \frac{j\delta}{n}} < \alpha. &
    \end{aligned}
    \label{gamma:ij}
  \end{align}
  Now define for all $1 \le i < j \le n$:\\
  \begin{minipage}{.49\textwidth}
  \begin{align*}
    h(i) & = g(i) + \sum_{j = 1}^{n} \varphi(i,j)\\
    h(i,j) & = g(i,j) - \varphi(i,j) \\
    h(I) & = g(I)  \mbox{, for all other $I$}
  \end{align*}
  \end{minipage}
  \begin{minipage}{.49\textwidth}
  \begin{align*}
   \varphi(i,j) =
     \begin{cases}
        0 & \text{if $i \to j$ and $j\to i$;} \\
        g(i,j) & \text{if $i\not\to j$ and $j\not\to i$;}\\
        \gamma(i,j) & \text{if $i\to j$ and $j\not\to i$.}
     \end{cases}
  \end{align*}
  \end{minipage}\\[2ex]
We only define  $h(i,j)$ when $i < j$.
  Note that $h^*(i) = g^*(i)$ for every $1 \le i \le n$, as
  we we add privately to $A_i$ as much as we remove from the private intersections $A_i \sqcap A_j$.

  We check the hypotheses of Lemma~\ref{lem:real} for $h$.
  First, if $i \to j$ and $j\to i$, then
  $$
  \begin{array}{lclclcl}
  \displaystyle{\frac{h^*(i,j)}{h^*(i)}} &= & \displaystyle{\frac{g^*(i,j)}{g^*(i)}}  &> &  \alpha.\\
  \end{array}
  $$
  Second, if $i \not\to j$ and $j\not\to i$, then
  $h^*(i,j) + g(i,j) = g^*(i,j)$.   Thus
    $$
  \begin{array}{lclclclcl}
  \displaystyle{\frac{h^*(i,j)}{h^*(i)}} &= &  \displaystyle{\frac{g^*(i,j) - g(i,j)}{g^*(i)}}
    &= &  \displaystyle{\frac{\alpha + \epsilon - \alpha(1-\alpha)^{n-2}}{1 + \epsilon + \frac{i\delta}{n}}}
      & < & \displaystyle{\frac{\alpha}{1 + \epsilon + \frac{i\delta}{n}}}
       & < & \alpha.
  \end{array}
  $$
  Finally,
    if $i \to j $ and $j \not\to i$, then $i<j$.    And by \eqref{gamma:ij},
    \begin{align*}
      \frac{h^*(i,j)}{h^*(i)}
      & =
       \frac{g^*(i,j) - \gamma(i,j)}{g^*(i)} > \alpha, \\
      \frac{h^*(i,j)}{h^*(j)}
      & =
       \frac{g^*(i,j) - \gamma(i,j)}{g^*(j)} < \alpha.
    \end{align*}
  This completes the proof.
\end{proof}

\begin{remark}
Let $G$ be a digraph on $n$ points with no one-way cycle.
If $\alpha \approx \frac{1}{2}$, then the method of Theorem~\ref{theorem-real} represents
a digraph $G$ a proportionality $\frac{1}{2}$-digraph with $|\bigcup A_g|
= O(n^2 2^{2n})$.  Here is the reasoning.
\begin{enumerate}\item
 Let $\alpha \sim 1/2$. Then  by Lemma 3.1, we have $f(I)=(1/2)^{n-1}$, for any $I$.

\item In the proof of Theorem 3.3.~with $\alpha \sim 1/2$, $\epsilon=(1/2)^{n+1}$ and $\delta=(1/2)^{n+2}$.

\item The maximum value of $g(I)$ is  at most $$\max\ f(I) + \epsilon~(1/2)^{n-1}+(1/2)^{n+1}=5(1/2)^{n+1}.$$

\item Next, we estimate size of $h(I)$.   It is less than
$$\max g(I)+n \times (\max \phi(i,j)) \sim 5(1/2)^{n+1}+n  (1/2)^{n-1}=(4n+5)(1/2)^{n+1}.$$

\item  Now $h(I)$ acts like the function $f(I)$ in Lemma 3.2. We take $N= 2^{2n}$. We have $|A(J)| < \max h(I) \times N \sim
(4n+5)2^{n-1}$.

\item We have $$| A(i)| < (2^n) \times (4n+5)2^{n-1}<(4n+5)2^{2n}.$$ Therefore, the size of $\bigcup A_i$ is
less than $(4n^2+5n)2^{2n}$.
\end{enumerate}

At the end of Section~\ref{section-rational}, we saw that the method of Theorem~\ref{theorem-rational}
represents $G$ as a majority digraph with $|\bigcup A_g| = O(n 2^n)$.

However, even though this suggests that  Theorem~\ref{theorem-real}  is not
as good a result as  Theorem~\ref{theorem-rational}
we emphasize that Theorem~\ref{theorem-real} works for all real $\alpha$.
We do not know how to extend the construction in  Theorem~\ref{theorem-rational}  to work on all real $\alpha$.
\end{remark}

\section{Application: the boolean  logic of ``most $X$ are $Y$''}
\label{section-logic}

\renewcommand{\most}[2]{\mbox{\sf M}({#1},{#2})}

We have characterized the
$\alpha$-propor\-tionality digraphs as those with no one-way cycle.
In particular, when $\alpha = 1/2$, we see that every digraph with no one-way cycle is a majority digraph.
We conclude with an application of this last result in logic.
What we discuss would be called a completeness theorem for the boolean  logic of
``most $X$ are $Y$''.   We start with a collection of \emph{one-place relation symbols} $X$, $Y$, $Z$, $\ldots$.
We then form \emph{atomic sentences} of the form $\most{X}{Y}$.
(Note that $X$ and $Y$ may be the same symbol here.   Up until now in this paper, we mainly worried about
such sentences when $X$ and $Y$ are \emph{different}.   So we have  a slight complication to keep in mind.)
$\most{X}{Y}$ is an abbreviation for {\sf Most $X$ are $Y$}.
Finally, we form \emph{sentences}
 from atomic sentences using the \emph{boolean connectives} of propositional logic,
namely negation ($\nott$), conjunction ($\andd$), disjunction ($\orr$), implication ($\iif$)
and bi-implication ($\iiff$).  So as just one example of a sentence, we would have
$$(\most{X}{Y} \andd \nott \most{X}{Z}) \orr \most{Y}{X}  \; .
$$
We call this logical language $\langmost$.
We are interested in the problem of \emph{inference} in $\langmost$.
To formulate this precisely, we need the notion of \emph{semantics}.
For this, we use \emph{models}.  A model of
$\langmost$ is a structure $\Model= (U, \semantics{ \ })$
 consisting of a finite set $U$
together with \emph{interpretations} $\semantics{X} \subseteq U$ of each one-place relation symbol $X$.
We then interpret our sentences in $\Model$ as follows
$$\begin{array}{lcl}
\Model\models \most{X}{Y} & \mbox{ iff } & |\semantics{X} \cap \semantics{Y}| > \frac{1}{2} |\semantics{X}| \; .\\
\end{array}
$$
We also read ``$\Model\models \most{X}{Y}$'' as  ``in the model $\Model$, most $X$'s are $Y$'s.''
If it is not the case that $\Model\models \most{X}{Y}$, then we write $\Model\not\models \most{X}{Y}$.
Observe that if  $\semantics{X}$ or $\semantics{Y}$
 is empty in a given model, then automatically $\Model\not\models \most{X}{Y} $.

We use $\phi$ and $\psi$   as variables  ranging over sentences in $\langmost$, and $\Gamma $ as a variable
denoting arbitrary finite sets of sentences.
Sentences with boolean connectives are given truth values  in the usual way.   For example,
$$\begin{array}{lcl}
\Model\models\nott \phi & \mbox{ iff } & \Model \not\models \phi\\
\Model\models \phi\andd\psi & \mbox{ iff } & \Model \models \phi  \mbox{ and }  \Model \models \psi\\
\end{array}
$$
We say that $\Model\models\Gamma$ if $\Model\models\psi$ for all $\psi\in\Gamma$.
The main semantic definition is:
$$\mbox{$\Gamma\models\phi$ if for all finite models $\Model$,
if $\Model\models\Gamma$, then $\Model\models\phi$.}
$$
This relation $\Gamma\models \phi$ between finite sets of sentences and single sentences is called the \emph{consequence relation}
of the logic.
Up until now, we have a semantic definition, having to do with all possible models of $\langmost$.
We shall define a proof-theoretic notion $\Gamma\proves \phi$ and then proving the
Soundness/Completeness Theorem:  $\Gamma\models\phi$ iff $\Gamma\proves\phi$.

One important feature of our semantics is that we restrict attention to \emph{finite} sets,
contrary to the usual practice in logic.   This is because we want to work with numerical proportions.

\begin{figure}[h]
\begin{mathframe}
$$\begin{array}{ll}
\mbox{Axioms} &  \mbox{all propositional tautologies} \\[1.5ex]
 & \mbox{$\most{X}{Y} \iif (\most{X}{X} \andd \most{Y}{Y})$}\\[1.5ex]
 &(\most{X_1}{X_2} \andd \most{X_2}{X_3}\andd \cdots \andd\most{X_n}{X_1} ) \\
&\qquad  \iif
(\most{X_2}{X_1} \orr \most{X_3}{X_2}\orr \cdots \orr \most{X_1}{X_n} )\\[1.5ex]
\mbox{Inference Rule} &  \mbox{from $\phi\iif\psi$ and $\phi$, infer $\psi$} \  (\mbox{Modus Ponens})\\[1.5ex]
\end{array} $$
\end{mathframe}
\caption{The logical system for $\langmost$, the  boolean logic of  ``most $X$ are $Y$''.
\label{fig-most}}
\end{figure}

The logical system that we use is defined in Figure~\ref{fig-most}.
By \emph{propositional tautologies} we mean   substitution instances of
propositional tautologies.  \rem{As examples, we have sentences such as
$$\begin{array}{l}
\most{X}{Y} \orr \nott \most{X}{Y} \\
(\most{X}{Y} \andd \most{Z}{Y}) \iif (\most{Z}{Y} \andd \most{X}{Y})
\end{array}
$$
The first of these is a substitution instance of the classical tautology $p\orr \nott p$, and
the second is an instance of $(p\andd q) \iif (q\andd p)$.
}
The next axiom just says that if $\most{X}{Y}$ in a given model, then $X$ and $Y$ must be non-empty
in the model. (Incidentally, in this discussion one should be sure to note the difference between two
uses of the $\to$ symbol: one for the edges in a digraph, and the other for a connective in $\langmost$.)
  Consequently, in the same model we have $\most{X}{X}$ and also $\most{Y}{Y}$.
This tells us that our axiom is \emph{sound}.

The key feature of the system  is the infinite collection of axioms
which together say that every cycle in the ``most'' relation has a reversal.    As we now know,
this characterizes majority digraphs.   Our logical result is
in essence a logical reformulation of this digraph-theoretic representation theorem.

We say that $\Gamma\proves \phi$ if there is a finite sequence of sentences such that each is
either a member of $\Gamma$ or an axiom, or else comes from earlier sentences in the sequence
using the one rule of the system, Modus Ponens.

As an example, the reader familiar with propositional logic will easily see that
$$\most{X}{Y}, \most{Y}{Z} \proves
\nott \most{Z}{X} \orr  \most{Y}{X}\orr  \most{Z}{Y} \orr  \most{X}{Z}.$$

\rem{
We shall show that
\begin{equation}
\most{X}{Y}, \most{Y}{Z} \proves
\nott \most{Z}{X} \orr  \most{Y}{X}\orr  \most{Z}{Y} \orr  \most{X}{Z}.
\label{ex-logic}
\end{equation}
Here is the formal derivation in our system:
$$\begin{nd}
\have{1} {\most{X}{Y}} \by{}{}
\have{1} {\most{Y}{Z}} \by{}{}
\have{12} {(\most{X}{Y} \andd  \most{Y}{Z} \andd \most{Z}{X})  \iif  (\most{Y}{X} \orr  \most{Z}{Y} \orr  \most{X}{Z}) }\by {}{}
\have{13} {\most{X}{Y} \iif (\most{Y}{Z}\iif (\phi\iif \psi ))} \by{}{}
\have{14} {\most{Y}{Z} \iif (\phi\iif   \psi)} \by{}{}
\have{20}{\phi \iif \psi} \by{}{}
\have{25} {\nott \most{Z}{X} \orr  \most{Y}{X}\orr  \most{Z}{Y} \orr  \most{X}{Z}} \by{}{}
\end{nd}
$$
In this, $\phi$ is the sentence on line 3, and $\psi$ is the sentence on line 7.
Lines $1$ and $2$ are our assumptions, the sentences on the left of the turnstile in (\ref{ex-logic}).
Line  $3$ is one of our axioms.
Line 4 is a (long and perhaps un-natural) propositional tautology.
Line 5 comes from lines 1 and 4 by Modus Ponens.
Line 6 comes from lines 2 and 5 by Modus Ponens.
Line 7 comes from lines 3 and 6 by Modus Ponens.
}

\begin{theorem}
For all finite sets $\Gamma\cup\set{\phi}$ of sentences in $\langmost$,
$\Gamma\proves\phi$ if and only if $\Gamma\models\phi$.
\label{theorem-completeness-boolean-most}
\end{theorem}

\begin{remark}
The  completeness half of this result is false if we allow $\Gamma$ to be infinite.   The reason is that
if we take
$$\begin{array}{lcl}
\Gamma & = &\set{ \most{X_{i+1}}{X_{i}}   : i = 1, 2, \ldots }
\cup
\set{\nott\most{X_i}{X_{i+1}} : i = 1, 2, \ldots },
\end{array}
$$
then we have $\Gamma\models\phi$ for all $\phi$,
 even when $\phi$ is a contradictory sentence such as $\most{X}{X}\andd\nott\most{X}{X}$.
   (To see this,
suppose that $\Model$ satisfies every sentence in $\Gamma$. Then
$|\semantics{$X_1$}| > |\semantics{$X_2$}| >  \cdots$.
It follows that there are no finite models of $\Gamma$.
And since our semantics is only concerned with finite models, it follows that $\Gamma\models\phi$ for all $\phi$.)
But for  a contradictory $\phi$, $\Gamma\not\proves\phi$: proofs are finite, and it is easy to see from
the soundness that no finite subset of $\Gamma$ can derive a contradiction.
\end{remark}

\begin{proof}
Using standard facts,
we may restrict attention to the
case when $\Gamma$ is the empty set.
In effect, we can move sentences across both relations $\models$ and $\proves$.
So we are left to prove that
$\proves\phi$ if and only if
$\models \phi$.  In words, $\phi$ has a proof in our system
if and only if  $\phi$ is true in all models.

The soundness part is a routine induction on the lengths of proofs in the system, and we are going
to omit these details.
In fact, soundness of a logical system is a very weak property, and the main point of interest is
the completeness of the system.
 We argue for an equivalent assertion: if $\phi$ is \emph{consistent} in the logic
 (that is, if $\not\proves\nott\phi$),
then there is some (finite) model of $\phi$.

Let $\mathcal{F}$ be the finite set of one-place relational symbols $X$, $Y$, $\ldots$, which occur in $\phi$.
 Using the propositional part of the logic, we may assume that our consistent sentence $\phi$
is in
  \emph{disjunctive normal form over $\FF$}.  That is, $\phi$ may be written as $\psi\orr\cdots \orr\psi_n$, where $n \geq 1$
  and
(1) each $\psi_i$ is a conjunction of atomic sentences and their negations;
(2) for all $X, Y\in \FF$, $\psi_i$ either contains $\most{X}{Y}$ as a conjunct, or else it contains
$\nott\most{X}{Y}$ as a conjunct;
(3) each $\psi_i$ is consistent in the logic.
We show that  $\psi$ has a model.  (The same holds for the other $\psi_i$.)
  Then a model of $\psi$ is a model of $\phi$, and we are done.

Let
$$G  \quadeq \set{X :   \most{X}{X} \mbox{ is a conjunct of $\psi$}},
$$
And make $G$
 into a simple digraph
by setting (for $X\neq Y$)
$$X\to Y \mbox{ in $G$} \quadiff   \most{X}{Y} \mbox{ is a conjunct of $\psi$}.
$$
We claim that  every cycle in $G$ has a reversal.  For suppose that
in $G$,
$$Z_1 \to Z_2 \to Z_3 \to \cdots \to Z_n\to Z_1\; .
$$
Then $\psi$ has conjuncts $\most{Z_1}{Z_2}$, $\ldots$, $\most{Z_n}{Z_1}$.
If $G$ had no reversal, then $\psi$ would also have as conjuncts
$\nott \most{Z_2}{Z_1}$, $\ldots$, $\nott\most{Z_1}{Z_n}$.
And using the logic, we would see that $\proves\nott\psi$; that is, $\psi$
would be inconsistent.    We conclude from this contradiction that indeed
every cycle in $G$ has a reversal.

By Theorem~\ref{theorem-rational},
$G$ is a majority digraph.   This gives finite sets $A_X$ for $X\in G$
with the property that for $X\neq Y$,
\begin{equation}
X\to Y \quadiff   |A_X\cap A_Y| > \frac{1}{2}|A_X|\; .
\label{done}
\end{equation}
and hence we get a
model: let $U= \bigcup_X  A_X$,
and  let $\semantics{X} = A_X$ when $X\in G$, and $\semantics{X} = \emptyset$ when $X\notin G$.

 We claim that
$\Model\models\psi$.   For a conjunct  of $\psi$ of the form $\most{X}{Y}$,
we argue as follows: the first axiom of the logic having to do with {\sf M} implies
that both $X$ and $Y$ belong to $G$.  And then the construction arranged that
$\Model\models \most{X}{Y}$.

Consider a conjunct  $\nott\most{X}{Y}$.   If both $X$ and $Y$ belong to $G$,
and if $X \neq Y$,
then the construction arranged that $\Model\models \nott \most{X}{Y}$.
If either $X$ or $Y$ is not in $G$, then $\semantics{X} = \emptyset$ or
$\semantics{Y} = \emptyset$, and again we have $\Model\models \nott\most{X}{Y}$.
If $X, Y \in G$ and $ X = Y$, then $\most{X}{Y}$ is a conjunct of $\psi$ by definition of $G$,
and we contradict the consistency of $\psi$.

This completes the proof.
\end{proof}

We conclude with a remark on the satisfiability problem for $\langmost$.
By this we mean the question of whether a given sentence $\phi$ of $\langmost$
has a model $\Model$ in our sense: a finite set $U$ and sets $\semantics{X}$  which make $\phi$ true
according to the definition.  Note that every model $\Model$ also gives us a truth assignment to the
 atomic sentences $\most{X}{Y}$ of $\langmost$.
It is convenient to regard these atomic  sentences $\most{X}{Y}$ as ``variables'' and construct propositional
logic over them.   When we do this, then every model gives a truth assignment to these ``variables.''

\begin{proposition}
The satisfiability problem for $\langmost$ is NP-complete.
\end{proposition}

\begin{proof}
Given a sentence $\phi$, one can guess an assignment $\alpha$ and verify that $\alpha$ both satisfies $\phi$
and also corresponds to a model in our sense.   This last point boils down to taking $\alpha$
and making a digraph $G_\alpha$ the way we did in the proof of Theorem~\ref{theorem-completeness-boolean-most}:
the vertices in $G_\alpha$ are the variables $X$ such that $\alpha(\most{X}{X}) = \mbox{true}$, and
$X\to Y$ in $G_\alpha$ iff $\alpha(\most{X}{Y}) = \mbox{true}$.   We can check in polynomial time that $G_{\alpha}$
has the property that every cycle has a reversal.

In the other direction, we reduce 3SAT to
our problem.
Suppose we are  given a 3SAT instance over a set $\set{x_1, \ldots, x_n}$ of boolean variables.
We are going to consider $\langmost$ formulated over a set of (twice as many) variables $X_1, \ldots, X_{2n}$.
Translate via $x_i\mapsto \most{X_{2i-1}}{X_{2i}}$.    For example,
a clause like $x_1 \orr \nott x_2 \orr x_3$ translates to
$$\most{X_1}{X_2} \orr \nott \most{X_3}{X_4} \orr \most{X_5}{X_6}.
$$
Translate a 3SAT instance $\alpha$ clause-by-clause in this way.   We only need to check that the translation
preserves  satisfiability; the converse is obvious.   If our original 3SAT instance were satisfiable,
we take a satisfying assignment $\alpha$ and convert it to a digraph $G_{\alpha}$ just as in our last paragraph.
The point is that the translation $\alpha^t$ arranges that
all of the edges in $G_{\alpha}$ are of the form $X_{2i-1} \to X_{2i}$ for some $i$.
The structure of $G_{\alpha}$ makes it trivially a majority digraph:
when $X_{2i-1} \to X_{2i}$ is an edge of $G_{\alpha}$,
let $A_{X_{2i-1}}$ be a singleton $\set{2i-1}$,
let $A_{X_{2i}}$ be this set $\set{2i-1}$ with two more points;
in all other cases, we take disjoint singletons.
A finite model corresponding to $G_{\alpha}$ satisfies $\alpha^t$.
\end{proof}

\section{Conclusion and Further Questions}

We have shown that a digraph $G$ with no one-way cycles is a proportionality $\alpha$-digraph for
all $\alpha\in (0,1)$.  But we do not know the smallest size of the sets $A_v$ or of their union,
as a function of $\alpha$ and $|G|$.

One could also
study digraphs which are representable by the ``exactly $\alpha$'' condition.   That is, given $\alpha\in (0,1)$,
which digraphs $G$ have the property that there are finite sets $A_v$ corresponding to the vertices of $G$
such that
  $u\to v$  in $G$ if and only if  $ |A_u \cap A_v | = \alpha \cdot |A_u|$?

For our last variations, suppose that $\alpha < \beta$ and that we ask of a digraph $G$ that there be
finite sets $A_v$ such that $u\to v$ in $G$ if and only if
 $ \alpha \cdot |A_u| < |A_u \cap A_v | < \beta \cdot |A_u|$.
 Let us call this condition \emph{$(\alpha,\beta)$-proportionality}.
 We do not know the exact characterization of
 the class of all $(\alpha,\beta)$-proportional digraphs.
One can show that if a digraph $G$ has no one-way cycles, then it is
 $(\alpha,\beta)$-proportional.    This is a corollary to
 the proof of   Theorem~\ref{theorem-real}
by  taking $\epsilon$ and $\delta$ sufficiently small,
all the numbers involved in Theorem~\ref{theorem-real}
will be so close to $\alpha$ that $\beta$ is irrelevant.
 But the converse is false: it is not necessary that a digraph have no one-way cycles
 in order for it to be  $(\alpha,\beta)$-proportional.
 For example,  take $\beta = .99$, $\alpha = .5$, and $G$ to be the
 one-way cycle $u\to v \to w\to u$.
 This digraph is
  $(\alpha,\beta)$-proportional:
  take
   $A_u = \set{ 0,1,2,3,4,5 }$,
  $A_v = \set{ 0,1,2,3,  6,7,8,9 }$,
and $A_w = \set{ 0,1,2,  6,7 }$.
Thus   $(\alpha,\beta)$-proportionality is weaker than the property of having no one-way cycles.
So we leave open the exact characterization.

Similarly, we would say that a digraph is \emph{$]\alpha,\beta[$-proportional}
if there are
finite sets $A_v$ such that
  $u\to v$ in $G$
  if and only if $|A_u \cap A_v | \le \alpha \cdot |A_u|$ or  $\beta \cdot |A_u| \le |A_u \cap A_v |$.
  Then $G$ is $(\alpha,\beta)$-proportional if and only if its complement $G^c$ is
  $]\alpha,\beta[$-proportional.   So the two concepts would have complementary characterizations.
  Again, we ask for a characterization of   $]\alpha,\beta[$-proportional digraphs.

There is
much more to be done on the logic of ``most'', since the language $\langmost$
of Section~\ref{section-logic}
was extremely limited: by adding interesting expressions to that language, one quickly arrives
at questions which seem interesting both from the viewpoints of logic and of combinatorics.
For a different contribution to this project, see~\cite{EM}.

\section*{Acknowledgment}  It is a pleasure to thank Ian Pratt-Hartmann for his interest in, and comments on, this paper.
We also thank several anonymous reviewers.

\bibliographystyle{amsplain}

\end{document}